\documentclass[12pt,a4paper,reqno]{amsart}

\usepackage{xcolor}
\usepackage[all,2cell,arrow,color]{xy}
\usepackage{amssymb}
\usepackage{mathtools}
\usepackage{tensor}
\numberwithin{equation}{section}
\usepackage{graphicx} 
\usepackage{rotating}
\usepackage{pdflscape}
\definecolor{darkgreen}{rgb}{0,0.45,0} 
\usepackage[pagebackref,colorlinks,citecolor=darkgreen,linkcolor=darkgreen]{hyperref}
\usepackage{enumerate,xspace}

\UseAllTwocells
\CompileMatrices

\def\black{\color{black}}
\definecolor{pinegreen}{rgb}{0.15,0.7,0.15}

\definecolor{lightgrey}{rgb}{0.666666,0.666666,0.666666}

\definecolor{brown}{rgb}{.44,.13,0}

\newcommand{\ssm}{\ensuremath{\mathsf M}\xspace}
\newcommand{\ssn}{\ensuremath{\mathsf N}\xspace}
\newcommand{\ssc}{\ensuremath{\mathsf C}\xspace}

\newcommand{\cc}{\ensuremath{\mathcal C}\xspace}
\newcommand{\cm}{\ensuremath{\mathcal M}\xspace}
\newcommand{\cq}{\ensuremath{\mathcal Q}\xspace}
\newcommand{\mM}{\ensuremath{\mathbb M}\xspace}
\newcommand{\op}{^{\mathsf{op}}}
\newcommand{\rev}{^{\mathsf{rev}}}
\newcommand{\nto}{\nrightarrow}

\newcommand{\Q}{\mathsf Q \black}

  \newtheorem{proposition}{Proposition}[section]

  \newtheorem{theorem}[proposition]{Theorem}

  \theoremstyle{definition}

\theoremstyle{remark}
  \newtheorem{remark}[proposition]{Remark}

  \newcounter{c}
  
  \newcommand{\etyk}[1]{\vspace{-7.4mm}$$\begin{equation}\Label{#1}
  \addtocounter{c}{1}}
  \renewcommand{\]}{\ifnum \value{c}=1 $$\else \end{equation}\fi}
\begin{document}

\title{A simplicial approach to multiplier bimonoids}

\author{Gabriella B\"ohm} 
\address{Wigner Research Centre for Physics, H-1525 Budapest 114,
P.O.B.\ 49, Hungary}
\email{bohm.gabriella@wigner.mta.hu}
\author{Stephen Lack}
\address{Department of Mathematics, Macquarie University NSW 2109, Australia}
\email{steve.lack@mq.edu.au}
\date\today

\begin{abstract}
Although multiplier bimonoids in general are not known to correspond to
comonoids in any monoidal category, we classify them in terms of maps from the 
Catalan simplicial set to another suitable simplicial set; thus they can be
regarded as (co)monoids in something more general than a monoidal 
category (namely, the simplicial set itself). 
We analyze the particular simplicial maps corresponding to that class of
multiplier bimonoids which can be regarded as comonoids.   
\end{abstract}
  
\maketitle


\section{Introduction} \label{sec:intro}
 
The recent paper \cite{BuckleyGarnerLackStreet} showed that monoids, as well
as many  generalizations, including monads, monoidal categories,
skew monoidal categories \cite{Szlach:skew}, and internal versions of these,
can be classified as simplicial maps from the {\em Catalan simplicial set} 
$\mathsf C$ to appropriately chosen simplicial sets. For the constructions in
the current paper the most relevant observation in
\cite{BuckleyGarnerLackStreet} is a bijective correspondence between
monoids in a monoidal category $\mathcal M$ and  simplicial maps
from $\mathsf C$ to the nerve $\mathsf N(\mathcal M)$ of $\mathcal M$.  

{\em Bialgebras} --- over a field or, more generally, in a braided monoidal
category \cc --- can be defined as comonoids in the monoidal category \cm of
monoids in \cc. Thus applying the results of \cite{BuckleyGarnerLackStreet},
they are classified by simplicial maps from $\mathsf C$ to the nerve $\mathsf
N(\mathcal M\op)$ of the category \cm regarded with the opposite composition. 

Classically, {\em Hopf algebras} over a field are defined as bialgebras
with a further property. {\em Multiplier Hopf algebras}
\cite{VanDaele:multiplier_Hopf} generalize Hopf algebras beyond the case when
the algebra has a unit. The typical motivating example of a multiplier Hopf
algebra consists of finitely supported functions on an infinite group with
values in the base field. The analogous notion of {\em multiplier bialgebra}
was introduced later,  in \cite{BohmGomezTorrecillasLopezCentella:wmba},
together with its `weak' generalization.

For many applications it is important to work with Hopf algebras and
bialgebras not only over fields but, more generally, in braided monoidal
categories which are different from the symmetric monoidal category of vector
spaces. The {\em formulation of multiplier bialgebras in braided monoidal
categories} is our longstanding project initiated in
\cite{BohmLack:braided_mba}.   

Generalizing some constructions in \cite{JanssenVercruysse:mba&mha} to
any braided monoidal category $\mathcal C$, we described in
\cite{BohmLack:cat_of_mbm} how {\em certain} multiplier bimonoids in $\mathcal
C$ can be seen as {\em certain} comonoids in an appropriately constructed
monoidal category $\mathcal M$. 
In light of the findings of \cite{BuckleyGarnerLackStreet}, this gives
rise to a  correspondence between these multiplier bialgebras and certain
simplicial maps from the Catalan simplicial set $\mathsf C$ to $\mathsf
N(\mathcal M\op)$.

The aim of this paper is to go beyond that characterization and prove a
bijection between  {\em arbitrary} multiplier bimonoids in
$\mathcal C$ and {\em arbitrary} simplicial maps from $\mathsf C$
to a suitable simplicial set $\mathsf M_{12}$. The simplicial maps $\mathsf C
\to \mathsf N(\mathcal M\op)$ that correspond, via the results of
\cite{BohmLack:cat_of_mbm} and \cite{BuckleyGarnerLackStreet}, to nice enough
multiplier bimonoids, turn out to factorize through a canonical 
embedding of a simplicial subset of $\mathsf M_{12}$ into  $\mathsf
N(\mathcal M\op)$.

{\em Regular} multiplier bimonoids constitute a distinguished class of
multiplier bimonoids. In order to classify them as well, we also present a
simplicial set $\mathsf M$ with the property that simplicial maps $\mathsf C
\to \mathsf M$ correspond bijectively to regular multiplier bimonoids in
\cc. 
\smallskip

\subsection*{Notation}
Throughout the paper, \cc denotes a braided monoidal category. We do not
assume that it is strict but --- relying on coherence --- we omit explicitly
denoting the associativity and unit isomorphisms. The monoidal unit is denoted
by $I$ and the monoidal product is denoted by juxtaposition (we also use the
power notation for the iterated monoidal product of the same object). The
braiding is denoted by $c$. The composite of morphisms $f\colon  A\to
B$ and $g\colon B\to C$ in \cc is denoted by $g.f\colon:A\to C$ and we write
$1$ for the identity morphisms in \cc. 
\smallskip

\subsection*{Acknowledgement}
We gratefully acknowledge the financial support of the Hungarian Scientific
Research Fund OTKA (grant K108384), as well as the Australian Research Council
Discovery Grant (DP130101969) and an ARC Future Fellowship (FT110100385). The
second named author is grateful for the warm hospitality of his hosts during
visits to the Wigner Research Centre in Sept-Oct 2014 and Aug-Sept 2015. 

\section{Preliminaries on the Catalan simplicial set}

In this section we briefly recall from \cite{BuckleyGarnerLackStreet} an
explicit description of the Catalan simplicial set and its role in the
classification of monads in bicategories; thus in particular of monoids in
monoidal categories.  

\subsection{Simplicial sets}

Consider the {\em simplex category} $\Delta$ whose objects are non-empty
finite ordinals and whose morphisms are the order preserving functions. By 
definition, a {\em simplicial set} is a presheaf on $\Delta$. Explicitly, a
simplicial set $\mathsf W$ is given by a collection $\{\mathsf W_n\}$ of sets
labelled by the natural numbers $n$ --- the sets of {\em $n$-simplices} --- 
together with the {\em face maps} $d_i\colon \mathsf W_n \to \mathsf W_{n-1}$
and the {\em degeneracy maps} $s_i\colon \mathsf W_n \to \mathsf W_{n+1}$, for
$0\leq i \leq n$, obeying the {\em simplicial relations}:
\begin{gather*}
d_i d_j = d_{j-1} d_i \quad \textrm{if\ \ } i<j \qquad\qquad
s_i s_j = s_{j+1} s_i \quad \textrm{if\ \ } i\leq j \\
d_is_j=\left\{
\begin{array}{ll}
s_{j-1} d_i \quad & \textrm{if\ \  } i<j \\
1 \quad & \textrm{if\ \  } i\in\{j,j+1\} \\
s_jd_{i-1} \quad & \textrm{if\ \ } i>j+1.
\end{array}
\right. 
\end{gather*}
An $n$-simplex is said to be {\em degenerate} if it belongs to the image of
one of the degeneracy maps, otherwise it is {\em non-degenerate}.

We often draw an $n$-simplex $w$ as an $n$-dimensional oriented 
geometric simplex whose $n-1$ dimensional faces are $d_i(w)$. For
example, for $n=2$ we draw  an oriented triangle 
$$
\xymatrix@C=15pt{
& w_1:=d_0d_2(w)=d_1d_0(w) \ar[rd]^-{w_{12}:=d_0(w)} \ar@{}[d]|-w \\
w_0:=d_1d_2(w)=d_1d_1(w) \ar[ru]^-{w_{01}:=d_2(w)} \ar[rr]_-{w_{02}:=d_1(w)} &&
w_2:=d_0d_1(w)=d_0d_0(w),}
$$
for $n=3$ we draw an oriented tetrahedron, and so on. 

A {\em simplicial map} is a natural transformation between such
presheaves. Explicitly, a simplicial map $\mathsf W \to \mathsf V$ is a
collection of functions $\{f_n\colon \mathsf W_n\to \mathsf V_n\}$ labelled by
the natural numbers $n$ which commute with the face and degeneracy maps in the
sense that $s_if_n=f_{n+1}s_i$ and $d_if_n=f_{n-1}d_i $ for all possible
values of $i$. 

\subsection{The Catalan simplicial set}

The {\em Catalan simplicial set} $\mathsf C$ has a single 0-simplex $\ast$. It
has two 1-simplices: $s_0(\ast)$ and a non-degenerate one to be called
$\alpha$. There are three degenerate 2-simplices $s_0s_0(\ast)=s_1s_0(\ast)$,
$s_0(\alpha)$ and $s_1(\alpha)$ and two non-degenerate ones to be denoted by 
$$
\xymatrix{
&\ast \ar[rd]^-{\alpha} \ar@{}[d]|-\tau \\
\ast \ar[ru]^-{\alpha} \ar[rr]_-\alpha &&
\ast}\qquad 
\xymatrix{
&\ast \ar[rd]^-{s_0(\ast)} \ar@{}[d]|-\varepsilon \\
\ast \ar[ru]^-{s_0(\ast)} \ar[rr]_-\alpha &&
\ast .}
$$
All higher simplices are generated {\em coskeletally}, meaning that for
any natural number $n>2$, and for any {\em $n$-boundary} (that is, $n+1$-tuple
$\{w_0,\dots,w_n\}$ of $n-1$-simplices such that $d_j(w_i)=d_i(w_{j+1})$ for
all $0\leq i \leq j <n$) there is a unique {\em filler} (that is, an
$n$-simplex $w$ obeying $d_i(w)=w_i$ for all $0\leq i \leq n$). In this 
situation we write $w=(w_0,\dots,w_n)$.

From this property of the Catalan simplicial set one can deduce that there
are four non-degenerate 3-simplices   
$$
\phi  = (\tau,\tau,\tau,\tau),\,
\lambda =(\varepsilon,s_1(\alpha),\tau,s_1(\alpha)),\,
\varrho =(s_0(\alpha),\tau,s_0(\alpha),\varepsilon),\,
\kappa =(\varepsilon,s_1(\alpha),s_0(\alpha),\varepsilon) 
$$
corresponding to the four tetrahedra drawn below. 
$$
\xymatrix@C=20pt@R=20pt{
\ast \ar[rr]^-\alpha \ar[dd]_-\alpha \ar[rrdd]|(.3)\alpha^(.8)\tau_(.8)\tau &&
\ast \ar[dd]^-\alpha \ar@{-}[ld]^(.4)\tau_(.4)\tau \\
& \ar[ld]|-\alpha \\
\ast &&
\ast \ar[ll]^-\alpha}\quad
\xymatrix@C=20pt@R=20pt{
\ast \ar[rr]^-\alpha \ar[dd]_-\alpha 
\ar[rrdd]|(.3)\alpha^(.65){\hspace{-5pt}s_1(\alpha)}_(.7){s_1(\alpha)\hspace{-5pt}} &&
\ast \ar[dd]^-{s_0(\ast)} \ar@{-}[ld]_(.4)\tau^(.4)\varepsilon \\
& \ar[ld]|-\alpha \\
\ast &&
\ast \ar[ll]^-{s_0(\ast)}}\quad
\xymatrix@C=20pt@R=20pt{
\ast \ar[rr]^-{s_0(\ast)} \ar[dd]_-\alpha 
\ar[rrdd]|(.3)\alpha^(.8)\varepsilon_(.8)\tau &&
\ast \ar[dd]^-{s_0(\ast)} 
\ar@{-}[ld]^(.7){\hspace{-5pt}s_0(\alpha)}_(.5){s_0(\alpha)\hspace{-5pt}} \\
& \ar[ld]|-\alpha \\
\ast &&
\ast \ar[ll]^-\alpha}\quad
\xymatrix@C=20pt@R=20pt{
\ast \ar[rr]^-{s_0(\ast)} \ar[dd]_-\alpha 
\ar[rrdd]|(.3)\alpha^(.8)\varepsilon_(.7){s_1(\alpha)\hspace{-5pt}}  &&
\ast \ar[dd]^-{s_0(\ast)} 
\ar@{-}[ld]^(.4)\varepsilon_(.5){s_0(\alpha)\hspace{-5pt}} \\
& \ar[ld]|-\alpha \\
\ast &&
\ast \ar[ll]^-{s_0(\ast)}}
$$
For some equivalent, more conceptual, descriptions of $\mathsf C$ consult
\cite{BuckleyGarnerLackStreet}. 

\subsection{The nerve of a bicategory} \label{sect:nerve}

Any bicategory $\mathcal B$ determines a simplicial set $\mathsf N(\mathcal
B)$ known as the {\em nerve} of $\mathcal B$. The 0-simplices of $\mathsf
N(\mathcal B)$ are the objects of $\mathcal B$. The 1-simplices are the
1-cells of $\mathcal B$, with faces provided by the source and the
target maps. For a given 2-boundary $\{w_{12},w_{02},w_{01}\}$, the 2-simplices 
$$
\xymatrix{
&w_1\ar[rd]^-{w_{12}} \ar@{}[d]|-w \\
w_0 \ar[ru]^-{w_{01}} \ar[rr]_-{w_{02}} &&
w_2}
$$
are 2-cells $w\colon w_{12}w_{01} \to w_{02}$ in $\mathcal B$. For a given 
3-boundary $\{w_{123},w_{023},w_{013},w_{012}\}$, there is precisely one filler
if the diagram 
$$
\xymatrix{
(w_{23}w_{12})w_{01} \ar[r]^-\cong \ar[d]_-{w_{123}1} &
w_{23}(w_{12}w_{01}) \ar[r]^-{1w_{012}} &
w_{23}w_{02} \ar[d]^-{w_{023}} \\
w_{13}w_{01} \ar[rr]_-{w_{013}} &&
w_{03}}
$$
commutes and there is no filler otherwise. If the filler exists then it is
denoted by $(w_{123},w_{023},$ $w_{013},w_{012})$. All higher simplices are
determined coskeletally. 

The degenerate 1-simplex on a 0-simplex $A$ is the identity 1-cell $1\colon
A\to A$. The degenerate 2-simplices $s_0( a )$ and
$s_1 ( a ) $ on a 1-simplex $a$ are the coherence
isomorphism 2-cells $a1\to a$ and $1a\to a$, respectively. On higher simplices
the degeneracy maps are determined by the uniqueness of the filler for a given
boundary. 

As was observed in \cite{BuckleyGarnerLackStreet}, a simplicial map $\mathsf
C \to \mathsf N(\mathcal B)$ is the same thing as a monad in $\mathcal B$. The
1-cell underlying the monad is the image of $\alpha$, with multiplication and
unit provided by the images of $\tau$ and $\varepsilon$, respectively. 

Monoidal categories can be seen as bicategories with a single object. Thus the
above considerations apply in particular to them.  In particular, the
above used symbol $\mathsf N(\cm)$ stands for the {\em monoidal nerve}
of a monoidal category \cm (rather than the nerve of the underlying ordinary
category). Since a monad in a one-object bicategory is the same as a monoid in
the corresponding monoidal category, simplicial maps $\mathsf C \to \mathsf
N(\cm)$ classify the monoids in \cm. 

\section{A simplicial description of multiplier bimonoids}

Multiplier bimonoids in braided monoidal categories are defined as compatible
pairs of counital fusion morphisms \cite{BohmLack:braided_mba}. Thus it
is not too surprising that the first step in our simplicial 
characterization of multiplier bimonoids is a simplicial treatment of
counital fusion morphisms.
In Section \ref{sect:mbm} we shall associate to the braided monoidal category
\cc a simplicial set $\mathsf M_{12}$ such that a simplicial map $\mathsf C\to
\mathsf M_{12}$ is the same thing as a multiplier bimonoid in \cc.   
As a preparation for that, first we construct in Section \ref{sect:fusion} a
simplicial set $\mathsf M_1$ and analyze the relation of simplicial maps
$\mathsf C \to \mathsf M_1$ to counital fusion morphisms in \cc. The
simplicial set $\mathsf M_1$ and its symmetric counterpart $\mathsf M_2$ will
be used as building blocks of $\mathsf M_{12}$.   

\subsection{Counital fusion morphisms}\label{sect:fusion}
Recall that a {\em fusion morphism} on an object $A$ of \cc is a morphism
$t\colon A^2 \to A^2$ making commutative the first diagram of  
\begin{equation}\label{eq:fusion_morphism}
\xymatrix@C=1pt{
A^3 \ar[rrr]^-{t1} \ar[d]_-{1t} &&&
A^3 \ar[rrr]^-{1t} &&&
A^3 \\
A^3 \ar[rr]_-{c1} &&
A^3 \ar[rr]_-{1t} &&
A^3 \ar[rr]_-{c^{-1}1} &&
A^3 \ar[u]_-{t1}} \qquad
\xymatrix{
A^2\ar[r]^-t \ar[rd]_-{1e} &
A^2 \ar[d]^-{1e} \\
& A.} 
\end{equation}
The morphism $e\colon A\to I$ is a {\em counit} of $t$ if it makes commutative
the second diagram of \eqref{eq:fusion_morphism}.  

The simplicial set $\mathsf M_1$ has a single 0-simplex $\ast$. Its
1-simplices are the semigroups in \cc; that is, objects $A$ equipped with an
associative multiplication $m\colon A^2\to A$. The 2-simplices with given
2-boundary $\{A_{12},A_{02},A_{01}\}$ are morphisms $\varphi\colon
A_{02}A_{12} \to A_{01}A_{12}$ in \cc rendering commutative the diagrams   
$$
\scalebox{.98}{
\xymatrix@C=5pt{
A_{02}A_{12}A_{12} \ar[r]^-{\varphi 1} \ar[d]_-{1m_{12}} &
A_{01}A_{12}A_{12} \ar[d]^-{1m_{12}} \\
A_{02}A_{12} \ar[r]_-{\varphi} &
A_{01}A_{12}}
\quad
\xymatrix@C=1pt{
A_{02}A_{02}A_{12} \ar[d]_-{1\varphi} \ar[rrr]^-{m_{02}1} &&&
A_{02}A_{12} \ar[rrr]^-\varphi &&&
A_{01}A_{12} \\
A_{02}A_{01}A_{12} \ar[rr]_-{c1} &&
A_{01}A_{02}A_{12} \ar[rr]_-{1\varphi} &&
A_{01}A_{01}A_{12} \ar[rr]_-{c^{-1}1} &&
A_{01}A_{01}A_{12}. \ar[u]_-{m_{01}1}}
}
$$
For a given 3-boundary
$\{\varphi_{123},\varphi_{023},\varphi_{013},\varphi_{012}\}$ there is precisely
one filler if the fusion equation 
\begin{equation}  \label{eq:fusion}
\xymatrix@C=5pt{
A_{03}A_{13}A_{23} \ar[d]_-{1\varphi_{123}} \ar[rrr]^-{\varphi_{013}1} &&&
A_{01}A_{13}A_{23} \ar[rrr]^-{1\varphi_{123}} &&&
A_{01}A_{12}A_{23} \\
A_{03}A_{12}A_{23} \ar[rr]_-{c1} &&
A_{12}A_{03}A_{23} \ar[rr]_-{1\varphi_{023}} &&
A_{12}A_{02}A_{23} \ar[rr]_-{c^{-1}1} &&
A_{02}A_{12}A_{23} \ar[u]_-{\varphi_{012}1}}
\end{equation}
commutes --- in which case we write
$(\varphi_{123},\varphi_{023},\varphi_{013},\varphi_{012})$ for the filler ---
and there is no filler otherwise. All higher simplices are generated
coskeletally.  
 
The action of the face maps should be clear from the above presentation.
The unique degenerate 1-simplex is the monoidal unit $I$ of \cc (regarded as a
trivial semigroup), while for a 1-simplex $(A,m)$ the degenerate 2-simplices
$s_0(A,m)$ and $s_1(A,m)$ are given by 
$$
\xymatrix{
& \ast \ar[rd]^-A \ar@{}[d]|-m\\
\ast \ar[ru]^-{I} \ar[rr]_-A &&
\ast}
\quad \raisebox{-20pt}{$\textrm{and}$}\quad 
\xymatrix{
& \ast \ar[rd]^-I \ar@{}[d]|-1\\
\ast \ar[ru]^-{A} \ar[rr]_-A &&
\ast}
$$
respectively. On higher simplices the degeneracy maps are determined by the
uniqueness of the filler for a given boundary.  

\begin{remark}\label{rmk:widget} 
If the simplicial set $\mathsf M_1$ looks contrived, we can motivate it as
follows, based on the fusion equation \eqref{eq:fusion}. Suppose we were to
try to define a simplicial set with a unique 0-simplex, with objects of \cc as
1-simplices, with morphisms of the form  $\varphi\colon A_{02}A_{12}\to
A_{01}A_{12}$ as 2-simplices, and with 4-tuples
$(\varphi_{123},\varphi_{023},\varphi_{013},\varphi_{012})$  of 
2-simplices  satisfying the fusion equation  as 3-simplices,
and with higher simplices defined coskeletally. 
 Then we can define the last degeneracy map in each degree as in
the definition of $\mathsf M_1$, but the other degeneracy maps will not
exist. For a 1-simplex $A$, the degenerate 2-simplex $s_0(A)$ should be
a morphism $AA\to IA$ in \cc, so it will be available if we require that,
instead of being mere objects $A$ of \cc, each 1-simplex be equipped
with a morphism $m\colon A^2\to A$. For a 2-simplex $\varphi$ as
above, the existence of 3-simplices with the boundary appropriate for
$s_0( \varphi )$ and $s_1( \varphi)$ amounts to the
commutativity of the two diagrams in the definition of 2-simplices in $\mathsf
M_1$. Furthermore the  morphism $AA\to IA$ induced by $m$ will
satisfy these equations if and only if $m$ is associative. Thus in this sense,
the  form of the simplicial set $\mathsf M_1$ is forced upon us by
the fusion equation.     

More formally, we could proceed as follows. Any morphism in the simplex
category $\Delta$ has an epi-mono factorization.
Considering only those morphisms in $\Delta$ in whose factorization only the
last fibre of the epimorphism has more than one element, we obtain a 
subcategory to be denoted by $\nabla$. The construction of the 
previous paragraph determines a presheaf $\mathcal P_\cc$ on $\nabla$.

The inclusion functor $J \colon
\nabla\to \Delta$ induces a functor $J^*\colon [\Delta\op,\mathsf{Set}] \to
[\nabla\op, \mathsf{Set}]$ between the presheaf categories, possessing a right
adjoint given by the right Kan extension $\mathsf{Ran}_J$. 
Explicitly, $J^*$ takes a simplicial set to the presheaf on $\nabla$ obtained
by forgetting all but the last degeneracy map; and $\mathsf{Ran}_J$ takes a
presheaf $\mathsf X$ on $\nabla$ to the simplicial set whose $n$-simplices are
families $\{x_f\in \mathsf X_j \}$ labelled by morphisms $f\colon j\to n$ in
$\Delta$ such that for all morphisms $w$ of codomain $j$ in $\nabla$, the
identity $w^*(x_f)=w_{f.w}$ holds (where $w^*$ denotes the image of $w$ under
the functor $\mathsf X\colon \nabla\op \to \mathsf{Set}$). 

The value of $\mathsf{Ran}_J$ at the presheaf $\mathcal P_\cc$
on $\nabla$ is precisely $\mathsf M_1$.   
\end{remark}

\begin{proposition} \label{prop:M1}
For any braided monoidal category \cc, consider the associated simplicial set
$\mathsf M_1$ above. To give a simplicial map $\mathsf C \to \mathsf M_1$ is
the same as specifying an object $A$ in \cc equipped with both a semigroup
structure with multiplication $m\colon A^2\to A$ and a fusion morphism
$t\colon A^2 \to A^2$ with counit $e\colon A\to I$, subject to the following
compatibility relations.  
$$
\xymatrix{
A^3 \ar[r]^-{t1} \ar[d]_-{1m} \ar@{}[rd]|-{\textrm{(a)}}&
A^3 \ar[d]^-{1m} \\
A^2 \ar[r]_-t &
A^2} 
\xymatrix@C=1pt{
A^3 \ar[rrr]^-{m1} \ar[d]_-{1t} \ar@{}[rrrrrrd]|-{\textrm{(b)}}&&&
A^2 \ar[rrr]^-t &&&
A^2 \\
A^3 \ar[rr]_-{c1} &&
A^3 \ar[rr]_-{1t} &&
A^3 \ar[rr]_-{c^{-1}1} &&
A^3\ar[u]_-{m1}} 
\xymatrix{
A^2 \ar[r]^-{e1} \ar[d]_-m \ar@{}[rd]|-{\textrm{(c)}}&
A \ar[d]^-e \\
A \ar[r]_-e &
I}
\xymatrix@C=10pt{
A^3 \ar[rr]^-{m1} \ar[d]_-{1m} \ar@{}[rrd]|-{\textrm{(d)}}&&
A^2 \ar[d]^-m \\
A^2 \ar[r]_-t &
A^2 \ar[r]_-{e1} &
A}
$$
\end{proposition}

\begin{proof}
A simplicial map $\mathsf C \to \mathsf M_1$ is given by the images of the
non-degenerate 1-simplex $\alpha$ and of the non-degenerate 2-simplices $\tau$
and $\varepsilon$. This means, respectively, a semigroup 
$(A,m)$, a morphism $t\colon A^2\to A^2$ in \cc making the 
diagrams (a) and (b) in the claim commute, and a morphism $e\colon A 
\to I$ making diagram (c) commute. The simplicial map can be 
defined on the non-degenerate 3-simplex $\phi$ of $\mathsf C$ if and 
only if $t$ obeys the fusion equation in the first diagram of 
\eqref{eq:fusion_morphism}. It can be defined on the 3-simplex 
$\lambda$ if and only if the counitality condition in the second 
diagram of \eqref{eq:fusion_morphism} holds. It can be defined on 
$\varrho$ if and only if diagram (d) of the claim commutes, while for 
$\kappa$ we get the same condition encoded in diagram (c).  
\end{proof}

\begin{remark}\label{rem:non-deg}
If $t\colon A^2 \to A^2$ is a fusion morphism with counit $e \colon A \to I$
in \cc, then we have a semigroup $(A,m:=e1.t)$ in \cc for which all diagrams of
Proposition \ref{prop:M1} commute: (a), (b) and (c) can be found in (3.6),
(3.5) and (3.4) in \cite{BohmLack:multiplier_Hopf_monoid}, respectively, and
(d) follows by the associativity of $m=e1.t$. Hence there is a corresponding
simplicial map $\mathsf C \to \mathsf M_1$.  

However, there may be more general simplicial maps $\mathsf C \to \mathsf M_1$
for which the multiplication of the corresponding semigroup $(A,m)$ is
different from the multiplication $e1.t$ coming from the counital fusion 
morphism $(A,t,e)$.

Let us consider the particular kind of simplicial maps $\mathsf C \to \mathsf
M_1$ for which the multiplication $m$ happens to be {\em non-degenerate} in
the sense that both maps 
\begin{eqnarray*}
&\cc(X,AY) \to \cc(AX,AY) \qquad
&f \mapsto m1.1f \\
&\cc(X,YA) \to \cc(XA,YA) \qquad
&g \mapsto 1m.g1 
\end{eqnarray*}
are injective, for any objects $X$ and $Y$. Since by identities (a) and (d) in
Proposition \ref{prop:M1} 
$$
m.e11.t1= e1.1m.t1 = e1.t.1m = m.m1,
$$
we conclude that in this case $m=e1.t$. 

By the associativity of $m$ and commutativity of (d), $m=e1.t$ also follows if
$1m \colon A^3 \to A^2$ is an epimorphism.
\end{remark}

By the same construction as above, we can associate a 
simplicial set to the monoidal category $\cc\rev$, obtained from \cc by reversing the
monoidal product and using the same braiding $c$. The opposite of 
this simplicial set is called $\mathsf M_2$. Explicitly,  $\mathsf M_2$
also has a single 0-simplex $\ast$ and the semigroups in $\cc$ as
1-simplices. The 2-simplices of a given 2-boundary $\{A_{12},A_{02},A_{01}\}$
are now morphisms $\psi \colon A_{01}A_{02} \to A_{01}A_{12}$ in \cc making the
diagrams 
$$
\scalebox{.98}{
\xymatrix@C=5pt{
A_{01}A_{01}A_{02} \ar[r]^-{1\psi} \ar[d]_-{m_{01}1} &
A_{01}A_{01}A_{12} \ar[d]^-{m_{01}1} \\
A_{01}A_{02} \ar[r]_-{\psi} &
A_{01}A_{12}}
\quad
\xymatrix@C=1pt{
A_{01}A_{02}A_{02} \ar[d]_-{\psi 1} \ar[rrr]^-{1 m_{02}} &&&
A_{01}A_{02} \ar[rrr]^-\psi &&&
A_{01}A_{12} \\
A_{01}A_{12}A_{02} \ar[rr]_-{1c} &&
A_{01}A_{02}A_{12} \ar[rr]_-{\psi 1} &&
A_{01}A_{12}A_{12} \ar[rr]_-{1c^{-1}} &&
A_{01}A_{12}A_{12} \ar[u]_-{1m_{12}}}}
$$
commute.
For a given 3-boundary
$\{\psi_{123},\psi_{023},\psi_{013},\psi_{012}\}$ there is precisely one filler if 
$$ 
\xymatrix@C=5pt{
A_{01}A_{02}A_{03} \ar[d]_-{\psi_{012}1} \ar[rrr]^-{1\psi_{023}} &&&
A_{01}A_{02}A_{23} \ar[rrr]^-{\psi_{012}1} &&&
A_{01}A_{12}A_{23} \\
A_{01}A_{12}A_{03} \ar[rr]_-{1c} &&
A_{01}A_{03}A_{12} \ar[rr]_-{\psi_{013}1} &&
A_{01}A_{13}A_{12} \ar[rr]_-{1 c^{-1}} &&
A_{01}A_{12}A_{13} \ar[u]_-{1\psi_{123}}}
$$
commutes --- in which case we write
$(\psi_{123},\psi_{023},\psi_{013},\psi_{012})$ for the filler ---
and there is no filler otherwise. All higher simplices are generated
coskeletally. The unique degenerate 1-simplex is again the monoidal unit $I$
of \cc (regarded as a trivial semigroup), while for a 1-simplex $(A,m)$ the
degenerate  2-simplices $s_0(A,m)$ and $s_1(A,m)$ are given by 
$$
\xymatrix{
& \ast \ar[rd]^-A \ar@{}[d]|-1\\
\ast \ar[ru]^-I \ar[rr]_-A &&
\ast} 
\quad \raisebox{-20pt}{$\textrm{and}$}\quad 
\xymatrix{
& \ast \ar[rd]^-I \ar@{}[d]|-m\\
\ast \ar[ru]^-A \ar[rr]_-A &&
\ast}
$$
respectively; note that the roles of $s_0$ and $s_1$ have been
interchanged. As before, on the higher simplices the degeneracy maps are 
determined by the uniqueness of the filler for a given boundary.

Since the simplicial set $\mathsf C$ is isomorphic to its opposite, 
Proposition \ref{prop:M1} then characterizes the simplicial 
maps $\mathsf C \to \mathsf M_2$ as objects $A$ carrying the 
compatible structures of a semigroup, and a counital fusion morphism 
in $\cc\rev$; once again, the roles of $\lambda$ and $\rho$ have been 
interchanged relative to the case of $\mathsf M_1$. 

As a further possibility, we can use the above construction to associate a
simplicial set $\mathsf M_3$ to the monoidal category $\overline \cc$,
obtained from \cc by keeping the same monoidal product but replacing the
braiding $c$ with $c^{-1}$, and using the twisted multiplication $m.c^{-1}$
for a 1-simplex (that is, semigroup) $(A,m)$, so that in particular 
the degenerate 2-simplex $s_0(A,m)$ is given by $m.c^{-1}$. 
Proposition \ref{prop:M1} can also be used to describe the simplicial 
maps $\mathsf C \to \mathsf M_3$.  

Finally, applying the above construction to the braided monoidal category
$(\overline \cc)\rev = \overline{\cc \rev}$ we obtain a simplicial set
$\mathsf M_4$. 

\subsection{Multiplier bimonoids} \label{sect:mbm}

A {\em multiplier bimonoid} \cite{BohmLack:braided_mba} in a braided monoidal
category \cc consists of a fusion morphism $t_1$ in \cc and a fusion morphism
$t_2$ in $\cc\rev$ with a common counit $e \colon A \to I$ such that the
diagrams 
\begin{equation}\label{eq:mbm}
\xymatrix{
A^3 \ar[r]^-{t_2 1} \ar[d]_-{1 t_1} &
A^3 \ar[d]^-{1 t_1} \\
A^3 \ar[r]_-{t_2 1} &
A^3}\qquad
\xymatrix{
A^2 \ar[r]^-{t_2} \ar[d]_-{t_1} &
A^2 \ar[d]^-{1e} \\
A^2 \ar[r]_-{e1} &
A}
\end{equation}
commute. Thus by Remark \ref{rem:non-deg}, it can be thought of as a pair of
simplicial maps $\mathsf C \to \mathsf M_1$ and $\mathsf C \to \mathsf M_2$
subject to compatibility conditions expressing the fact that the underlying
semigroups and the counits are equal, and the diagrams in \eqref{eq:mbm}
commute, with the common diagonal of the second of these given by the
multiplication. Guided by this fact, we construct below a simplicial set 
$\mathsf M_{12}$ whose simplices are suitably compatible pairs consisting of a
simplex in $\mathsf M_1$ and a simplex in $\mathsf M_2$. We prove that a
simplicial map $\mathsf C \to \mathsf M_{12}$ is the same thing as a
multiplier bimonoid in \cc.  

The simplicial set $\mathsf M_{12}$ has a single 0-simplex $\ast$ and the
semigroups of \cc as 1-simplices. The 2-simplices are pairs $(\varphi\vert
\psi)$ consisting of a 2-simplex $\varphi$ of $\mathsf M_1$ and a 2-simplex
$\psi$ of $\mathsf M_2$ with common boundary $\{A_{12},A_{02},A_{01}\}$, such
that the diagram 
\begin{equation}\label{eq:M_12_2-simplex}
\xymatrix{
A_{01}A_{02}A_{12} \ar[r]^-{1 \varphi} \ar[d]_-{\psi 1} &
A_{01}A_{01}A_{12} \ar[d]^-{m_{01}1} \\
A_{01}A_{12}A_{12} \ar[r]_-{1m_{12}} &
A_{01} A_{12}}
\end{equation}
commutes. The 3-simplices are pairs
$(\varphi_{123},\varphi_{023},\varphi_{013},\varphi_{012} \vert 
\psi_{123},\psi_{023},\psi_{013},\psi_{012})$ consisting of a 3-simplex
$(\varphi_{123},\varphi_{023},\varphi_{013},\varphi_{012})$ in $\mathsf M_1$
and a 3-simplex $(\psi_{123},\psi_{023},\psi_{013},\psi_{012})$ in $\mathsf
M_2$  such that $(\phi_{ijk}\vert\psi_{ijk})$ is a 2-simplex in $\mathsf
M_{12}$ for each $ijk$, and the diagram
\begin{equation}\label{eq:M_12_3-simplex}
\xymatrix{
A_{01}A_{03}A_{23} \ar[r]^-{1 \varphi_{023}} \ar[d]_-{\psi_{013} 1} &
A_{01}A_{02}A_{23} \ar[d]^-{\psi_{012}1} \\
A_{01}A_{13}A_{23} \ar[r]_-{1\varphi_{123}} &
A_{01} A_{12}A_{23}}
\end{equation}
commutes. The face and degeneracy maps act on the pairs componentwise, and
higher simplices are defined coskeletally. 

\begin{remark} \label{rem:M_12}
We explained in Remark~\ref{rmk:widget} a sense in which the simplicial set
$\ssm_1$ is dictated by the fusion equation; in particular, the associativity
of the multiplications in the 1-simplices and the commutativity of the
diagrams in the definition of the 2-simplices are required in order for
various degenerate 3-simplices to satisfy the fusion equation. 

The case of $\ssm_{12}$ is similar: 
it has the same 1-simplices as $\ssm_1$, and once we impose
commutativity of \eqref{eq:M_12_3-simplex} on the 3-simplices, any 2-simplex $(\varphi\vert \psi)$ must obey
\eqref{eq:M_12_2-simplex} in order to have a 3-simplex with the boundary of $s_1(\varphi\vert \psi)$.  
\end{remark}

\begin{theorem} \label{thm:M12}
There is a bijection between simplicial maps $\mathsf C \to \mathsf M_{12}$
and multiplier bimonoids in \cc. 
\end{theorem}

\begin{proof}
Again, a simplicial map $\mathsf C \to \mathsf M_{12}$ is given by the images
$(A,m)$ of the non-degenerate 1-simplex $\alpha$, $(t_1\vert t_2)$ of the
non-degenerate 2-simplex $\tau$ and $(e_1\vert e_2)$ of the non-degenerate
2-simplex $\varepsilon$. By Proposition \ref{prop:M1} $(t_1,e_1)$ is a
counital fusion morphism in \cc obeying conditions (a)-(d); and $(t_2,e_2)$ is
a counital fusion morphism in $\cc\rev$ obeying symmetric counterparts of
conditions (a)-(d). Furthermore, there are compatibility conditions between
them: $(t_1\vert t_2)$ is a 2-simplex of $\mathsf M_{12}$ if and only if 
diagram (e) in
$$
\xymatrix{
A^3 \ar[r]^-{1 t_1} \ar[d]_-{t_2 1} \ar@{}[rd]|-{\textrm{(e)}} &
A^3 \ar[d]^-{m1}\\
A^3\ar[r]_-{1m} &
A^2}
\xymatrix{
A \ar[r]^-{e_1} \ar[d]_-{e_2} \ar@{}[rd]|-{\textrm{(f)}} &
I \ar@{=}[d] \\
I \ar@{=}[r] &
I}
\xymatrix{
A^3 \ar[r]^-{1 t_1} \ar[d]_-{t_2 1} \ar@{}[rd]|-{\textrm{(g)}} &
A^3 \ar[d]^-{t_2 1}\\
A^3\ar[r]_-{1t_1} &
A^3}
\xymatrix{
A^2 \ar@{=}[r] \ar[d]_-{t_2} \ar@{}[rd]|-{\textrm{(h)}} &
A^2 \ar[d]^-{m}\\
A^2\ar[r]_-{1e_1} &
A}
\xymatrix{
A^2 \ar@{=}[d] \ar[r]^-{t_1} \ar@{}[rd]|-{\textrm{(i)}} &
A^2 \ar[d]^-{e_2 1}\\
A^2\ar[r]_-{m} &
A}
$$
commutes and $(e_1\vert e_2)$ is a 2-simplex of $\mathsf M_{12}$ if and only
if (f) does so.
The simplicial map is well-defined on the non-degenerate 3-simplices $\phi$,
$\lambda$, $\varrho$ and $\kappa$ if and only if the respective diagrams (g),
(h), (i) and (f) again commute.

From (f) we infer that the counits $e_1$ and $e_2$ are equal so we will denote
them simply by $e$. Then (h) and (i) take the equivalent form in the second
diagram of \eqref{eq:mbm}, with common diagonal $m$.
As observed in Remark \ref{rem:non-deg}, from this it follows that all
identities (a)-(d), as well as their symmetric counterparts, hold
true. Diagram (g) is identical to the first diagram of \eqref{eq:mbm}; this 
implies (e) upon postcomposing by $1e1$. 

Summarizing, a simplicial map $\mathsf C \to \mathsf M_{12}$ is the same thing
as a pair of counital fusion morphism $(t_1,e)$ in \cc and a counital fusion
morphism $(t_2,e)$ in $\cc\rev$ (with common counit $e$) rendering commutative
the diagrams of \eqref{eq:mbm}.
\end{proof}

Applying the construction of this section to the braided monoidal category
$\overline \cc$, and using the reversed multiplications $m.c^{-1}$ of
the semigroups $(A,m)$, we obtain a simplicial set $\mathsf M_{34}$.

\subsection{Regular multiplier bimonoids}

A {\em regular multiplier bimonoid} \cite{BohmLack:braided_mba} in a braided
monoidal category $\cc$ is a tuple $(A,t_1,t_2,t_3,t_4,e)$ such that
$(A,t_1,t_2,e)$ is a multiplier bimonoid in \cc and $(A,t_3,t_4,e)$ is a
multiplier bimonoid in $\overline \cc$, and such that the following diagrams
commute, where $m$ stands for the common diagonal of the first diagram.   
\begin{equation}\label{eq:reg_mbm}
\xymatrix{
A^2 \ar[r]^-{t_1} \ar[d]_-c &
A^2 \ar[dd]^-{e1}
&
A^3 \ar[r]^-{1t_1} \ar[d]_-{c1} &
A^3 \ar[d]^-{c1}
&
A^3 \ar[r]^-{1t_1} \ar[dd]_-{t_41} &
A^3 \ar[dd]^-{t_41}
&
A^3 \ar[r]^-{t_21} \ar[d]_-{1c} &
A^3 \ar[d]^-{1c} 
&
A^3 \ar[r]^-{t_21} \ar[dd]_-{1t_3} &
A^3 \ar[dd]^-{1t_3} \\
A^2 \ar[d]_-{t_3} &
&
A^3 \ar[d]_-{t_31} &
A^3 \ar[d]^-{1m}
&
&&
A^3 \ar[d]_-{1t_4} &
A^3 \ar[d]^-{m1} \\
A^2 \ar[r]_-{e1} &
A
&
A^3 \ar[r]_-{1m} &
A^2
&
A^3 \ar[r]_-{1t_1} &
A^3
&
A^3 \ar[r]_-{m1} &
A^2
&
A^3 \ar[r]_-{t_21} &
A^3}
\end{equation}
We shall classify regular multiplier bimonoids via simplicial maps from
$\mathsf C$ to a simplicial set $\mathsf M$ which we now describe.

The simplicial set $\mathsf M$ has a single 0-simplex $\ast$, and its
1-simplices are the semigroups in the braided monoidal category \cc. The
2-simplices are pairs $(\varphi\vert \psi \vert\!\vert 
\varphi' \vert \psi')$ consisting of a 2-simplex $(\varphi \vert \psi)$ of
$\mathsf M_{12}$ and a 2-simplex $(\varphi' \vert \psi')$ of $\mathsf M_{34}$
with common boundary $(A,B,C)$, obeying the compatibility conditions 
$$
\xymatrix@C=25pt{
CBA \ar[r]^-{1\varphi} \ar[dd]_-{\psi' 1} &
C^2A \ar[d]^-{c^{-1} 1} \\
& 
C^2A
\ar[d]^-{m 1} \\
CA^2 \ar[r]_-{1m} &
CA}
\xymatrix@C=25pt{
CBA \ar[r]^-{\psi 1} \ar[dd]_-{1 \varphi'} &
CA^2 \ar[d]^-{1 c^{-1}} \\
& 
CA^2 \ar[d]^-{1 m} \\
C^2A \ar[r]_-{m 1} &
CA}
\xymatrix@C=25pt{
ABA \ar[r]^-{1\varphi} \ar[d]_-{c 1} &
ACA \ar[d]^-{c 1} \\
BA^2 \ar[d]_-{\varphi' 1} & 
CA^2 \ar[d]^-{1 m} \\
C A^2 \ar[r]_-{1m} &
CA}
\xymatrix@C=25pt{
CBC \ar[r]^-{\psi 1} \ar[d]_-{1 c} &
CAC \ar[d]^-{1 c} \\
C^2B \ar[d]_-{1 \psi'} & 
C^2A \ar[d]^-{m 1} \\
C^2A \ar[r]_-{m1} &
CA.}
$$
The 3-simplices are pairs 
$$(\varphi_{123},\varphi_{023},\varphi_{013},\varphi_{012} \vert \psi_{123},
\psi_{023},\psi_{013},\psi_{012} \vert \! \vert \varphi'_{123}, \varphi'_{023}, 
\varphi'_{013}, \varphi'_{012} \vert
\psi'_{123},\psi'_{023},\psi'_{013},\psi'_{012})$$   
consisting of a 3-simplex
$(\varphi_{123},\varphi_{023},\varphi_{013},\varphi_{012} \vert
\psi_{123},\psi_{023},\psi_{013},\psi_{012})$ of $\mathsf M_{12}$ and a 
3-simplex $(\varphi'_{123}, \varphi'_{023},\varphi'_{013},\varphi'_{012} \vert
\psi'_{123},\psi'_{023}, \psi'_{013}, \psi'_{012})$ of $\mathsf M_{34}$ for
which $(\varphi_{ijk}\vert \psi_{ijk} \vert\!\vert \varphi'_{ijk} \vert
\psi'_{ijk})$ is a 2-simplex in $\mathsf M$ for every $0\leq i<j<k\leq 3$ and
the diagrams    
\begin{equation}\label{eq:M_3-simplex}
\xymatrix{
A_{01}A_{03}A_{23} \ar[r]^-{1\varphi_{023}} \ar[d]_-{\psi'_{013} 1} &
A_{01}A_{02}A_{23} \ar[d]^-{\psi'_{012} 1} \\
A_{01}A_{13}A_{23} \ar[r]_-{1\varphi_{123}} &
A_{01}A_{12}A_{23}}
\qquad
\xymatrix{
A_{01}A_{03}A_{23} \ar[r]^-{\psi_{013} 1} \ar[d]_-{1 \varphi'_{023}} &
A_{01}A_{13}A_{23} \ar[d]^-{1 \varphi'_{123}} \\
A_{01}A_{02}A_{23} \ar[r]_-{\psi_{012} 1} &
A_{01}A_{12}A_{23}}
\end{equation}
commute. The higher simplices are generated coskeletally and the 
face and the degeneracy maps act on the pairs memberwise.

\begin{theorem}\label{thm:M}
There is a bijection between simplicial maps $\mathsf C \to \mathsf M$ and
regular multiplier bimonoids in \cc.
\end{theorem}

\begin{proof}
For a simplicial map $\mathsf C \to \mathsf M$, denote the image of the
2-simplex $\alpha$ by $(A,m)$, and denote the images of the 3-simplices $\tau$
and $\varepsilon$ by $(t_1\vert t_2 \vert\!\vert t_3 \vert t_4)$ and $(e\vert
e \vert\!\vert e' \vert e')$, respectively. By Theorem \ref{thm:M12},
$(t_1,t_2,e)$ is a multiplier bimonoid in \cc for which $m=e1.t_1=1e.t_2$, and
$(t_3,t_4,e')$ is a multiplier bimonoid in $\cc\rev$ for which
$m.c^{-1}=e'1.t_3=1e'.t_4$. Furthermore, from the requirements that $(t_1\vert
t_2 \vert\!\vert t_3 \vert t_4)$ and $(e\vert e \vert\!\vert e' \vert e')$ be
2-simplices of $\mathsf M$ we obtain the following identities.   
$$
\xymatrix{
A^3 \ar[r]^-{1t_1} \ar[dd]_-{t_41} \ar@{}[rdd]|-{\textrm{(j)}} &
A^3 \ar[d]^-{c^{-1}1} \\
& A^3 \ar[d]^-{m1} \\
A^3 \ar[r]_-{1m} &
A^2}
\xymatrix{
A^3 \ar[r]^-{t_2 1} \ar[dd]_-{1 t_3} \ar@{}[rdd]|-{\textrm{(k)}} &
A^3 \ar[d]^-{1 c^{-1}} \\
& A^3 \ar[d]^-{1 m} \\
A^3 \ar[r]_-{m1} &
A^2}
\xymatrix{
A^3 \ar[r]^-{1t_1} \ar[d]_-{c1} \ar@{}[rdd]|-{\textrm{(l)}} &
A^3 \ar[d]^-{c 1} \\
A^3 \ar[d]_-{t_31} & 
A^3 \ar[d]^-{1m} \\
A^3 \ar[r]_-{1m} &
A^2}
\xymatrix{
A^3 \ar[r]^-{t_2 1} \ar[d]_-{1 c} \ar@{}[rdd]|-{\textrm{(m)}} &
A^3 \ar[d]^-{1 c} \\
A^3 \ar[d]_-{1t_4} & 
A^3 \ar[d]^-{m1} \\
A^3 \ar[r]_-{m1} &
A^2}
\xymatrix@R=68pt{
A \ar[r]^-e \ar[d]_-{e'} \ar@{}[rd]|-{\textrm{(n)}} &
I \ar@{=}[d] \\
I \ar@{=}[r] &
I.}
$$
The boundaries of the images of the 3-simplices $\phi$, $\lambda$, $\varrho$
and $\kappa$ under a simplicial map are determined. The fillers to be
their images exist if and only if the diagrams
$$
\xymatrix@R=30pt @C=30pt{
A^3 \ar[r]^-{1t_1} \ar[d]_-{t_4 1} \ar@{}[rd]|-{\textrm{(o)}} &
A^3 \ar[d]^-{t_4 1} \\
A^3 \ar[r]_-{1t_1} &
A^3} \ \ 
\xymatrix@R=30pt @C=30pt{
A^3 \ar[r]^-{t_2 1} \ar[d]_-{1 t_3} \ar@{}[rd]|-{\textrm{(p)}} &
A^3 \ar[d]^-{1 t_3} \\
A^3 \ar[r]_-{t_2 1} &
A^3} \ \ 
\xymatrix@R=30pt @C=30pt{
A^2 \ar[r]^-{t_1} \ar[d]_-{t_2} 
\ar[rd]|(.35)m ^(.6){\textrm{(q)}}_(.6){\textrm{(r)}} &
A^2 \ar[d]^-{e' 1} \\
A \ar[r]_-{1e'} &
A} \ \ 
\xymatrix@R=30pt @C=30pt{
A^2 \ar[r]^-{t_3} \ar[d]_-{t_4} 
\ar[rd]|(.35){\ m.c^{-1}} ^(.6){\textrm{(s)}}_(.6){\textrm{(t)}}&
A^2 \ar[d]^-{e 1} \\
A \ar[r]_-{1 e} &
A}
$$
commute (in the case of $\kappa$ the same condition (n) occurs again). 

Conditions (l), (o), (m) and (p) are identical to the last four diagrams of 
\eqref{eq:reg_mbm}. In light of (n), conditions (q), (r), (s) and (t)
are redundant. Conditions (j) and (k) are also redundant:
they follow from (o) and (p), respectively, postcomposing them with $1e1=1e'
1$. Finally (n) implies the commutativity of the first diagram of
\eqref{eq:reg_mbm}, with common diagonal $m$.  
\end{proof}

\begin{remark} 
We discussed in Remark~\ref{rmk:widget} and Remark ~\ref{rem:M_12} a sense in
which the definitions of 2-simplices in $\mathsf M_1$ and $\mathsf M_{12}$ are
dictated by the definitions of the respective 3-simplices. 

When it comes to \ssm, however, this property breaks down. While commutativity
of the first two diagrams in the definition of 2-simplices in \ssm is needed
in order for degenerate 3-simplices to exist in \ssm, commutativity of the
other two diagrams is not. Commutativity of these last two diagrams would be
needed, though, if we were to modify the definition of 3-simplices so as to
require that, in addition to the diagrams of \eqref{eq:M_3-simplex},
also 
\begin{gather*}
\xymatrix@C=5pt{
A_{13}A_{03}A_{23} \ar[rr]^-{1\varphi_{023}} \ar[d]_-{c1} &&
A_{13}A_{02}A_{23} \ar[rr]^-{c1} &&
A_{02}A_{13}A_{23} \ar[rr]^-{1\varphi_{123}} &&
A_{02}A_{12}A_{23} \ar[d]^-{\varphi'_{012} 1} \\
A_{03}A_{13}A_{23} \ar[rrr]_-{\varphi'_{013} 1} &&&
A_{01}A_{13}A_{23} \ar[rrr]_-{1\varphi_{123}} &&&
A_{01}A_{12}A_{23}}
\\
\xymatrix@C=5pt{
A_{01}A_{03}A_{02} \ar[rr]^-{\psi_{013}1} \ar[d]_-{1c} &&
A_{01}A_{13}A_{02} \ar[rr]^-{1c} &&
A_{01}A_{02}A_{13} \ar[rr]^-{\psi_{012}1} &&
A_{01}A_{12}A_{13} \ar[d]^-{1\psi'_{123}} \\
A_{01}A_{02}A_{03} \ar[rrr]_-{1\psi'_{023}} &&&
A_{01}A_{02}A_{23} \ar[rrr]_-{\psi_{012}1} &&&
A_{01}A_{12}A_{23}}
\end{gather*}
commute. 

Simplicial maps from \ssc to the resulting simplicial set would now correspond
to a stronger notion of regular multiplier bimonoid, in which the second
diagram of \eqref{eq:reg_mbm} is replaced by the fusion equation in the second
diagram of \cite[Remark 3.10]{BohmLack:braided_mba}, and with an analogous
change to the fourth diagram of \eqref{eq:reg_mbm}. Although in  {\em general}
this would result in a stronger notion of regular multiplier bimonoid, the
difference would disappear in the case where the multiplication is
non-degenerate, and it was already anticipated in
\cite[Remark~3.10]{BohmLack:braided_mba} that in the not necessarily
non-degenerate case such strengthenings might be needed.  
\end{remark}

\subsection{Multiplier bimonoids which are comonoids}

In our paper \cite{BohmLack:cat_of_mbm}, following some ideas in
\cite{JanssenVercruysse:mba&mha}, we associated to any braided monoidal
category \cc a monoidal category \cm, and we described a correspondence
between certain multiplier bimonoids in \cc  and certain comonoids in \cm
\cite[Theorem~5.1]{BohmLack:cat_of_mbm}. In this section, we explain this
correspondence in terms of simplicial maps and the Catalan simplicial set.  

A comonoid in the monoidal category \cm is the same as a monoid in the
monoidal category $\cm\op$. We now form the nerve $\mathsf N(\cm\op)$ of the
monoidal category $\cm\op$, as in Section~\ref{sect:nerve}; this is not to be
confused with the nerve of the underlying category of $\cm\op$.  
As explained in \cite{BuckleyGarnerLackStreet},
simplicial maps ${\mathsf c\colon}\ssc\to\ssn(\cm\op)$ can
be identified with monoids in $\cm\op$, and so in turn with comonoids in \cm. 

On the other hand, we have shown that simplicial maps ${\mathsf
a\colon}\ssc\to\ssm_{12}$ can be identified with multiplier bimonoids in
\cc. In order to compare these, we construct a simplicial set
$\Q$ which is contained in both $\ssm_{12}$ and $\ssn(\cm\op)$. Now a
multiplier bimonoid  in \cc corresponds to a comonoid in \cm just when there
is a common factorization  of the corresponding simplicial maps as in the
following diagram.   
$$\xymatrix{
\ssc \ar@/^1pc/[drr]^{\mathsf a} \ar@/_1pc/[ddr]_{\mathsf c} 
\ar@{.>}[dr] \\
& \Q \ar[r] \ar[d] & \ssm_{12} \\
& \ssn(\cm\op) }$$

We now turn to the details. To define the monoidal category \cm,
in \cite{BohmLack:cat_of_mbm} we fixed a class \cq of regular
epimorphisms in \cc which is closed under composition and monoidal product,
contains the isomorphisms, and is right-cancellative in the sense that if
$s\colon A \to B$ and $t.s \colon A \to C$ are in \cq, then so is $t \colon B
\to C$. Since each $q \in \cq$ is a regular epimorphism, it is the coequalizer
of some pair of morphisms. Finally we suppose that this pair may be chosen in
such a way that the coequalizer is preserved by taking the monoidal product
with any fixed object.
 
The objects of the associated category \cm are those semigroups in \cc whose
multiplication is non-degenerate and belongs to \cq. The morphisms $f\colon
A\nto B$ are pairs $(f_1\colon AB \to B \leftarrow BA\colon f_2)$ of morphisms
in \cq such that the first two, equivalently, the last two diagrams in
$$
\xymatrix{
A^2B \ar[r]^-{1f_1} \ar[d]_-{m1} &
AB \ar[d]^-{f_1} \\
AB \ar[r]_-{f_1} &
B}\qquad 
\xymatrix{
BAB \ar[r]^-{1f_1} \ar[d]_-{f_2 1}&
B^2 \ar[d]^-m \\
B^2 \ar[r]_-m &
B} \qquad
\xymatrix{
BA^2 \ar[r]^-{f_2 1} \ar[d]_-{1m} &
BA \ar[d]^-{f_2} \\
BA \ar[r]_-{f_2} &
B}
$$
commute. The composite $g\bullet f$ of morphisms $f\colon A\nto B$ and
$g\colon B\nto C$ is defined by universality of the coequalizer in the top row
of either diagram in
$$
\xymatrix{
ABC \ar[r]^-{1g_1} \ar[d]_-{f_11} &
AC \ar[d]^-{(g\bullet f)_1} \\
BC \ar[r]_-{g_1} &
C}\qquad
\xymatrix{
CBA \ar[r]^-{g_2 1} \ar[d]_-{1 f_2} &
CA \ar[d]^-{(g\bullet f)_2} \\
CB \ar[r]_-{g_2} &
C.}
$$
The identity morphism $A\nto A$ is the pair $(m\colon A^2  \to A \leftarrow
A^2\colon m)$.

This category \cm is monoidal. The monoidal product of semigroups $A$ and
$C$ is 
$$
\xymatrix{
(AC)^2 \ar[r]^-{1c1} &
A^2C^2 \ar[r]^-{mm} &
AC}
$$
and the monoidal product of morphisms $f\colon A \nto B$ and $g\colon C\nto D$
is the pair 
$$ 
\xymatrix{
ACBD \ar[r]^-{1c1} &
ABCD \ar[r]^-{f_1g_1} &
BD &
BADC \ar[l]_-{f_2g_2} &
BDAC. \ar[l]_-{1c1} }
$$

If \cc is a closed braided monoidal category with pullbacks, then for any
semigroup $B$ with non-degenerate multiplication one can define its multiplier
monoid $\mM(B)$, see \cite{BohmLack:cat_of_mbm}. It is a monoid
in \cc and a universal object characterized by the property that morphisms
$ (f_1,f_2) \colon A\nto B$ in \cm correspond bijectively to
multiplicative morphisms $f\colon A\to \mM(B)$ in \cc such that  
$$
\xymatrix{
AB \ar[r]^-{f1} &
\mM(B)B \ar[r]^-{i_1} &
B}
\quad \textrm{and}\quad
\xymatrix{
BA \ar[r]^-{1f} &
B\mM(B) \ar[r]^-{i_2} &
B}
$$
are in \cq, where $i\colon \mM(B) \to \mM(B)$ is the identity morphism in \cc;
regarded as a morphism $(i_1,i_2)\colon \mM(B) \nto B$ in \cm.
In the category of vector spaces $\mM(B)$ reduces to the multiplier algebra
 of $B$ as defined in \cite{Dauns:Multiplier}. 

Take a multiplier bimonoid $(A,t_1,t_2,e)$ in \cc for which
\begin{itemize}
\item the  underlying semigroup has a non-degenerate
  multiplication $m:=e1.t_1=1e.t_2$  
\item $m,e$, and the morphisms $d_1$ and $d_2$ defined by
$$\xymatrix @R=6pt {
A^3 \ar[r]^-{c1} &
A^3 \ar[r]^-{1t_1} &
A^3 \ar[r]^-{c^{-1}1} &
A^3 \ar[r]^-{m1} &
A^2 \\
A^3 \ar[r]^-{1c} & A^3 \ar[r]^-{t_21} & A^3 \ar[r]^-{1c^{-1}} & 
A^3 \ar[r]^-{1m} & A^2}$$
all belong to \cq. 
\end{itemize}
The correspondence in \cite[Theorem 5.1]{BohmLack:cat_of_mbm} 
associated a multiplier bimonoid of this type to the comonoid in \cm with
underlying object $(A,m)$, with comultiplication $A\nto A^2$ having components $d_1$ and $d_2$, and with counit $A\nto I$ whose
components are both $e$. 

The simplicial set $\mathsf M_{12}$ has a simplicial subset $\Q$ as
follows. The only 0-simplex $\ast$ of $\mathsf M_{12}$ is a 0-simplex also in
$\Q$. The 1-simplices of $\Q$ are those semigroups $(A,m)$ in \cc whose 
multiplication is non-degenerate and belongs to \cq. The 2-simplices of
$\Q$ are those 2-simplices $(\varphi\vert \psi)$ of $\mathsf M_{12}$ whose
faces belong to $\Q$ and for which the morphisms 
$$
\xymatrix@R=8pt{
{\widehat \varphi}:A_{02} A_{01} A_{12} \ar[r]^-{1c^{-1}} &
A_{02} A_{12} A_{01} \ar[r]^-{\varphi 1} &
A_{01} A_{12} A_{01} \ar[r]^-{1c} &
A_{01} A_{01} A_{12} \ar[r]^-{m_{01} 1} &
A_{01} A_{12} \\
{\widehat \psi}:A_{01} A_{12} A_{02} \ar[r]^-{c^{-1} 1} &
A_{12} A_{01} A_{02} \ar[r]^-{1 \psi} &
A_{12} A_{01} A_{12} \ar[r]^-{c1} &
A_{01} A_{12} A_{12} \ar[r]^-{1 m_{12}} &
A_{01} A_{12}
}
$$
are in \cq. The 3-simplices of $\Q$ are all those 3-simplices of $\mathsf
M_{12}$ whose faces belong to $\Q$. Clearly a simplicial map from $\mathsf C$
to $\Q$ is the same thing as a multiplier bimonoid in \cc having the
properties listed above.  
 
The desired simplicial map $\Q \to \mathsf N(\cm\op)$ sends the 0-simplex
$\ast$ to the single 0-simplex of the nerve $\mathsf N(\cm\op)$. 
It sends the 1-simplex $(A,m)$ in $\mathsf Q$ to its underlying object 
$A$. It sends a 2-simplex $(\varphi\vert \psi)$ to the 
morphism $A_{02} \nto A_{01} A_{12}$ in \cm whose components are
$\widehat \varphi$ and $\widehat \psi$. On the higher simplices it is
unambiguously defined by the uniqueness of the filler of any boundary in
$\mathsf N(\cm\op)$. This assignment is injective by
non-degeneracy of the relevant multiplications. 



\bibliographystyle{plain}

\end{document}